\documentclass[12pt,reqno]{amsart}

\usepackage[dvips]{graphicx} 
\usepackage[arrow,matrix,curve]{xy}

\usepackage{amssymb,latexsym, amsmath, amscd, array, hyperref}

\usepackage{breakurl}  

\usepackage{tikz}     

                            {\egroup\par\bigskip}

\def\tag#1#2{\hbox to\textwidth{#1\hfil$\displaystyle #2$\hfil}}

\newtheorem{theorem}{Theorem}[section]
\newtheorem{lemma}[theorem]{Lemma}
\newtheorem{proposition}[theorem]{Proposition}
\newtheorem{corollary}[theorem]{Corollary}

\theoremstyle{definition}
\newtheorem{definition}[theorem]{Definition}

\newcommand{\RP}{{\mathbb R\mathbb P}}

\newcommand\CP {{\mathbb C\mathbb P}}

\newcommand{\HH}{{\mathbb H}}

\DeclareMathOperator{\sys}{{\rm sys}}

\DeclareMathOperator{\dR}{{\rm dR}}

\DeclareMathOperator{\Z}{\mathbb{Z}}

\DeclareMathOperator{\R}{{\mathbb R}}

\newcommand\C {{\mathbb C}}

\numberwithin{equation}{section}

\numberwithin{figure}{section} \numberwithin{table}{section}

\newcommand\stsys{{\rm stsys}}

\DeclareMathOperator{\vol}{{\rm vol}} 

\DeclareMathOperator{\area}{{\rm area}}

\def \cqfd{\unskip\kern 6pt\penalty 500
\raise -2pt\hbox{\vrule\vbox to10pt{\hrule width 4pt
\vfill\hrule}\vrule}\par}                 

\def\adots{\mathinner{\mkern2mu\raise1pt\hbox{.}
\mkern3mu\raise4pt\hbox{.}\mkern1mu\raise7pt\hbox{.}}}
\def\hfl#1{{\buildrel{#1}\over
{\hbox to 12mm{\rightarrowfill}}}}

\def \\R^n \times \R^n
\rightarrow \R{\mathop{\R^n \times \R^n
\rightarrow \R}}

\long\def\forget#1\forgotten{}

\begin{document}

\title{Extending Gromov's optimal systolic inequality}

\author{Thomas G. Goodwillie}\address{Department of Mathematics, Brown
  University} \email{tomg@math.brown.edu}

\author{James J. Hebda} \address{Department of Mathematics and
  Statistics, Saint Louis University, 220 N. Grand Boulevard,
  St. Louis, MO 63103, USA} \email{james.hebda@slu.edu}

\author{Mikhail G. Katz} \address{Department of Mathematics, Bar Ilan
  University, Ramat Gan 52900 Israel} \email{katzmik@math.biu.ac.il}

\begin{abstract}
The existence of nontrivial cup products or Massey products in the
cohomology of a manifold leads to inequalities of systolic type, but
in general such inequalities are not optimal (tight).  Gromov proved
an \emph{optimal} systolic inequality for complex projective space.
We provide a natural extension of Gromov's inequality to manifolds
whose fundamental cohomology class is a cup product of 2-dimensional
classes.
\end{abstract}

\maketitle

\tableofcontents

\section{Introduction}

Whenever the fundamental cohomology class of a manifold~$M$ decomposes
as a nontrivial cup product, there exist associated stable systolic
inequalities providing upper bounds on products of stable systolic
invariants of metrics on~$M$ in terms of the total volume.%
\footnote{Gromov \cite{Gr2}, \cite{Gr99}.}
Furthermore, when cup products are trivial, the existence of suitable
nonzero Massey products similarly leads to inequalities of systolic
type.%
\footnote{Katz and Lescop \cite{Ka05}; Katz \cite{Ka08}.}
Dranishnikov, Katz, and Rudyak introduced a numerical topological
invariant of a manifold~$M$, in terms of the existence of systolic
inequalities on~$M$, called its \emph{systolic category}, an invariant
that turned out to be closely related to its Lusternik--Shnirelmann
category.%
\footnote{See \cite{Ka06}, \cite{Dr08}, \cite{Dr11}.}
In some cases, the systolic category provides a better lower bound for
the Lusternik--Shnirelmann category than the cuplength.

However, most such inequalities are not optimal.  One exception is
Gromov's stable systolic inequality for complex projective~$n$-space,
where all Riemannian metrics satisfy~$\frac{\stsys_2^n}{\vol}\leq n!$.
Here the boundary case of equality is satisfied by the Fubini--Study
metric.  We provide a natural extension of Gromov's inequality to
$2$-essential manifolds (Theorem~\ref{c23} below).

There is only a handful of optimal systolic inequalities.  Loewner
proved around 1950 that all metrics on the~$2$-torus satisfy the bound
$\frac{\sys_1^2}{\area}\leq \sqrt{\tfrac43}$, where the boundary case
of equality occurs for the flat torus~$\C/L$ where~$L$ is the
hexagonal lattice (lattice of Eisenstein integers).  Pu \cite{Pu52}
proved in 1952 that all metrics on the real projective plane satisfy
the bound~$\frac{\sys_1^2}{\area}\leq \frac\pi2$, with equality
precisely for metrics of constant Gaussian curvature.  Bavard
\cite{Ba86} proved an optimal systolic inequality for the Klein
bottle, and described the (singular) metric realizing the boundary
case of equality.  Katz and Sabourau \cite{Ka15} proved an optimal
systolic inequality for metrics of nonpositive Gaussian curvature on
Dyck's surface~$\RP^2\#\RP^2\#\RP^2$.  Jabbour and Sabourau
\cite{Ja22} proved an optimal systolic inequality for closed geodesics
in complete metrics on the punctured sphere (with $3$ or $4$
punctures).  Additionally, Loewner's torus inequality extends to all
hyperelliptic surfaces~\cite{Ka06b}.  Recent advances in systolic
geometry include \cite{Ka24a} and \cite{Ka24b}.

\noindent
\begin{table}
\renewcommand{\arraystretch}{1.5}
\def\drawing {
\begin{tabular}[t]  
{|
@{\hspace{3pt}}p{0.5in}|| 
@{\hspace{3pt}}p{0.6in}|
@{\hspace{3pt}}p{0.28in}|
@{\hspace{3pt}}p{0.55in}|
@{\hspace{3pt}}p{0.7in}|
@{\hspace{3pt}}p{0.7in}|
@{\hspace{3pt}}p{0.72in}|
} 
\hline  & Loewner   & Pu & Gromov & Bavard & Katz Sabourau & Jabbour Sabourau
\\ 
\hline\hline
surface &
$T^2$  &
$\RP^2$ &
$\CP^n$ &
$\RP^2\#\RP^2$ &
$3\RP^2$ &
$S^2 - 3,4$
\\
\hline 
bound &
$\frac2{\sqrt3}$ &
$\frac\pi2$ & 
$n!$  \mbox{(stable)} 
& 
$\frac{\pi}{2\sqrt2}$ 
&
$\frac13(\sqrt2+1)$
$\phantom{a}(K\leq0)$
& 
$2{\sqrt3}$,
$\frac4{\sqrt3}$ (geodesics)
\\
\hline
\end{tabular}
\renewcommand{\arraystretch}{1}
}
\drawing
\caption{\textsf{Optimal systolic
    inequalities${}^{\phantom{I}}_{\phantom{I}}$}}
\label{tt}
\end{table}

\section{Lattices in dimension 2}

Proofs of stable systolic inequalities typically depend on analysis of
lattices in homology and cohomology groups, with respect to suitable
norms (see Section~\ref{s1014}).  Let~$L$ be a lattice in a Banach
space~$(B,\|\;\|)$ of dimension~$b$.  Let~$1\leq k \leq b$.
The~\emph{$k$-th successive minimum} of~$L$,
denoted~$\lambda_k(L)=\lambda_k(L,\|\;\|)$, is the least
number~$\lambda$ such that there exists a linearly
independent~$k$-tuple~$(x_1,\ldots,x_k)$ of elements in~$L$
satisfying~$\|x_i\|\leq \lambda$ for all~$i=1,\ldots,k$.  

\begin{definition}
\label{d10141}
For~$b>0$, let~$\Gamma_b>0$ be the supremum
of~$\lambda_1(L)\,\lambda_b(L^\ast)$ over all lattices~$L$ in
all~$b$-dimensional Banach spaces, where~$L^\ast$ is the dual lattice
in the dual Banach space.
\end{definition}

Mahler's compactness theorem implies that~$\Gamma_b$ is finite.  

Let~$\Gamma_b^e$ be the supremum of the
product~$\lambda_1(L)\,\lambda_b(L^\ast)$ over all lattices
in~$b$-dimensional Euclidean space.

\begin{proposition}
\label{p21}
We have~$\Gamma_b\le\sqrt b\;\Gamma_b^e$.
\end{proposition}

\begin{proof}
Let~$(B,\|\;\|)$ be a~$b$-dimensional Banach space.  By John's
ellipsoid theorem,%
\footnote{John \cite[p.\;203]{Jo48}; Milman and Schechtman
  \cite[Section 3.3, p.\,10]{Mi86}.}
there exists a Euclidean norm~$| \;|$ on~$B$ such that
$|\;|\le\|\;\|\le\sqrt b \, |\;|$.  The first of these two
inequalities implies~$\lambda_1(L,|\; |)\le \lambda_1(L,\|\;\|)$.  The
second implies~$|\;|^\ast \le \sqrt b\, \|\;\|^\ast$ and therefore
$\lambda_b(L^\ast,|\;|)\le \sqrt b\,\lambda_b(L^\ast,\|\;\|^\ast)$.
Together these
imply~$\lambda_1(L)\,\lambda_b(L^\ast)\le\sqrt{b}\;\Gamma_b^e$.
\end{proof}

Of course~$\Gamma_1=1=\Gamma_1^e$.  

\begin{proposition}
[Speyer]
\label{p24}
For~$2$-dimensional lattices, the maximal value of
$\lambda_1(L)\,\lambda_2(L^\ast)$ is~$\frac32$.%
\footnote{Speyer \url{https://mathoverflow.net/a/449498/28128} }
\end{proposition}

It is easy to show that $\Gamma_2^e=\frac2{\sqrt3}$.  Incidentally,
Mahler \cite{Ma48} (as reported in \cite[p.\;581]{Ka88}) showed
that~$\lambda_1(L)\, \lambda_1(L^\ast)\leq\sqrt2$.

It is known that~$\Gamma_b^e$ grows linearly in~$b$.%
\footnote{Banaszczyk \cite{Ba93}.}
Therefore by Proposition~\ref{p21},~$\Gamma_b=O\left(b^{1.5}\right)$.
Banaszczyk obtains a better bound~$O\left(b\log b\right)$.%
\footnote{Banaszczyk \cite{Ba95}.}

\section{Stable systolic inequality extending Gromov's}
\label{s1014}

Let~$M$ be a closed manifold.  Let~$L_k(M)$ be the image
of~$H_k(M;\Z)$ in~$H_k(M;\R)$.  Given a Riemannian metric on~$M$, the
stable norm~$\|C\|$ of a class~$C\in H_k(M;\R)$ is
~$\|C\|=\inf\big\{\vol_k(\tilde C) \colon \tilde C \in C \big\}$,
where the infimum is taken over all smooth real cycles.  The
stable~$k$-systole of~$M$ is by definition
\[
\stsys_k(M)=\lambda_1(L_k(M), \|\;\|).%
\footnote{Without the stabilisation (i.e., allowing denominators in
  cycles), one witnesses a widespread phenomenon of systolic freedom;
  see e.g., Babenko and Katz \cite{Ba98}.}
\]

Let~$L^k(M)\subseteq H^k_{\dR}(M)$ be the image of integral cohomology
in de Rham cohomology, the dual lattice of~$L_k(M)$.  The norm on
$H^k_{\dR}(M)$ that is dual to the stable norm on homology is the
comass norm.%
\footnote{Whitney \cite{Wh57}; Federer \cite[Section\;4.10,
    p.\;380]{Fe2}; Gromov \cite[Section 4.34, p.\;261]{Gr99}; Pansu
  \cite[Lemma\;17]{Pa99}.}
We recall the definition.  For a differential~$k$-form~$a$ at a point
$p$, the comass~$\|a_p\|$ is the maximum of~$a(v_1,...,v_k)$ over all
vectors~$v_i\in T_p M$ with~$|v_i|=1$.  The comass of~$a\in \Omega^k
(M)$ is~$\|a\|_\infty =\sup \left\{ \|a_p\| \colon p\in M\right\}$.
For~$\alpha\in H^k_{\dR}(M)$ the comass is
\[
\|\alpha\|^\ast = \inf_{a\in \alpha} \|a\|_\infty.
\]

\begin{lemma}
\label{l31}
The comass of a wedge of~$n$~$2$-forms of unit comass is at most~$n!$.
\end{lemma}

\begin{proof}
It suffices to prove a pointwise statement.  To simplify notation, we
will use single bars $|\;|$ for the comass in this proof.  Define the
norm~$|\omega|$, for a~$p$-form on~$\mathbb R^m$, to be the least
upper bound of~$|\omega(v_1,\dots ,v_p)|$ for vectors~$v_i$ of length
one. For any~$p$ and~$q$, let~$C_{p,q}$ be the smallest number such
that for every~$p$-form~$a$ and every~$q$-form~$b$ on~$\mathbb R^m$ we
have~$|a\wedge b|\le C\,|a|\,|b|$.  It does not depend on~$m$, because
if~$C$ is valid for forms in~$\mathbb R^{p+q}$ and
if~$v_1,\dots,v_{p+q}$ span~$V\cong \mathbb R^{p+q}$ in~$\mathbb R^m$
then
$$ |(a\wedge b)(v_1,\ldots ,v_{p+q})|\le
C\,|a_V|\,|b_V|\,|v_1|\cdots |v_{p+q}|\le
C\,|a|\,|b|\,|v_1|\cdots |v_{p+q}|,$$ where~$a_V$
and~$b_V$ are the restrictions to~$V$.

We will show that~$C_{2,2n-2}\le n$. It follows that for~$2$-forms
$\omega_1,\dots,\omega_n$ we have~$|\omega_1\wedge \dots \wedge
\omega_n|\le n\,|\omega_1\wedge \dots \wedge
\omega_{n-1}|\,|\omega_n|$, and therefore by induction on~$n$ we
have~$|\omega_1\wedge\cdots\wedge\omega_n|\le n!\,|\omega_1|\cdots
|\omega_n|$.

To prove that~$C_{2,2n-2}\le n$, first observe that~$C_{p,q}$ has the
following alternative interpretation: it is the smallest~$C$ such that
the inner product of two~$p$-forms on~$\mathbb R^{p+q}$ always
satisfies~$a\cdot b\le C\,|a|\,|b|$. This is so because the norm
satisfies~$|b|=|\ast b|$ (Hodge star) and~$|a\wedge \ast b|=|a\cdot
b|$.

So a restatement of the claim to be proved is that for~$2$-forms
on~$\mathbb R^{2n}$ we have~$a\cdot b\le n\,|a|\,|b|$. This can be
seen by making a change of orthogonal basis so that
$$a = a_1e_1\wedge e_2+\dots+ a_ne_{2n-1}\wedge e_{2n}$$ 
for some~$a_1,\dots ,a_n$. Write 
$$b=b_1e_1\wedge e_2+\dots+ b_ne_{2n-1}\wedge e_{2n}+\dots ,$$ where
the other terms will not matter. Then
$$|a|=\max|a_i|$$
$$|b|\ge \max |b_i|$$
\[
a\cdot b=\sum_i a_ib_i\le n\,\max|a_ib_i|\le n\,|a|\,|b|.
\]

To explain why this norm is invariant under Hodge star, suppose
that~$b$ is a~$p$-form in~$\mathbb R^{p+q}$.  The claim is
that~$|(\ast b)(v_1,\dots,v_q)|\le |b|\,|v_1|\dots|v_q|$. Without loss
of generality the vectors~$v_i$ are orthogonal and of length
one. Complete them to an orthonormal basis by
vectors~$w_1,\dots,w_p$. Then~$|(\ast b)(v_1,\dots,v_q)|= |b(w_1,\dots
,w_p)|=|b|$, as required.%
\footnote{See also Goodwillie
  \url{https://mathoverflow.net/a/449042/28128} The bound also results
  by representing the cup product by a suitable Pfaffian, and applying
  Roos \cite[Lemma~2.1, p.\,1788]{Ro15} (see there for some history;
  the result goes back to Banach).}
\end{proof}

In Gromov's 1981 book,%
\footnote{Gromov \cite[item 4.37, p.\;60]{Gr81}.}
one finds the following comment in the paragraph discussing
Wirtinger's inequality and the optimal stable systolic inequality
for~$\CP^n$: ``$(2n)!/2^n$ {\ldots} est la meilleure constante pour la
comasse d'un produit de~$n$~$2$-formes quelconques.''  This was
translated as follows in the English edition:%
\footnote{Gromov \cite[item 4.37, p.\;262]{Gr99}.}
\begin{quote}
$(2n)!/2^n$ {\ldots} is the best constant for the comass of the
  product of~$n$ arbitrary~$2$-forms.
\end{quote}
Note that~$(2n)!/2^n$ is considerably larger than~$n!$.

\begin{definition}
A closed orientable manifold is~$2$-\emph{essential} if its
fundamental cohomology class is expressible as a cup product
of~$2$-dimensonal classes.
\end{definition}

\begin{theorem}
\label{c23}
All metrics on a~$2$-essential manifold~$M$ of dimension~$2n$ satisfy
the bound
\[
\frac{\stsys_2^n}{\vol}\leq n!\,(\Gamma_b)^n,
\]
where~$b=b_2(M)$ is the second Betti number.
\end{theorem}

\begin{proof}
In~$L^2_{\dR}(M)$, choose a linearly independent spanning set such
that the comass norm of each element is at
most~$\lambda_{b}\left(L^{2\phantom{I}}_{\dR}(M),\|\;\|^\ast\right)$.
Since~$M$ is~$2$-essential, by linearity there exist~$\alpha_j$ in
this spanning set such that~$\alpha_1\cup\,\cdots\,\cup\alpha_n$ is a
non-zero (integer) multiple of the fundamental cohomology class.

Let~$a_j$ be~$2$-forms representing the integer classes~$\alpha_j$.
Recall that the comass norm of~$\alpha_j$ is the infimum of comass
norms of representative~$2$-forms~$a_j$.  By Lemma~\ref{l31}, we
obtain
\begin{equation}
\label{e31}
1\leq\Big|\int_M a_1\wedge\cdots\wedge a_n\Big| \leq
n!\,\|\alpha_1\|^\ast \cdots \|\alpha_n\|^\ast\vol(M).
\end{equation}
We now multiply both sides of \eqref{e31} by the~$n$-th power of
$\stsys_2=\lambda_1(L_2(M),\|\;\|)$ to obtain
\[
\begin{aligned}
(\stsys_2)^n &\leq
  n!\,\lambda_1(L_2(M),\|\;\|)^n\left\|\alpha_1\right\|^\ast\cdots
  \|\alpha_n\|^\ast\vol(M) \\&\leq
  n!\left[\lambda_1(L_2(M),\|\;\|)\;\lambda_{b}\!
    \left(L^2_{\dR}(M),\|\;\|^\ast\right) \right]^n\vol(M) \\&\leq
  n!\,(\Gamma_{b})^n\vol(M),
\end{aligned}
\]
as required.
\end{proof}

When~$b_2(M)=1$ we have~$\Gamma_{b}=1$ and the inequality
specializes to Gromov's stable systolic inequality for the complex
projective space.

The following is immediate from Proposition~\ref{p24}.

\begin{corollary}
All metrics on a~$2$-essential~$(2n)$-manifold~$M$ with
$b_2(M)=2$ satisfy the bound
\[
\frac{\stsys_2^n}{\vol}\leq n!\,\big(\tfrac32\big)^n.
\]
\end{corollary}

Obtaining optimal inequalities for higher stable systoles appears to
be difficult.  Thus, for the quaternionic projective plane~$\HH
\mathbb P^2$, the analysis of the constant in the stable systolic
inequality involves an analysis of~$4$-forms on~$\R^8$.  All metrics
on the quaternionic projective plane~$\HH \mathbb P^2$ satisfy the
inequality~$\frac{\stsys_4^2}{\vol}\leq14$, but the optimal constant
is only known to be in the interval~$[6,14]$.%
\footnote{Bangert et al.~\cite[Proposition~1.4]{Ba09}).}
The symmetric metric is not optimal in this case, and has a systolic
ratio of only~$\frac{10}3$.

\section{Acknowledgements}

The authors are grateful to P\'eter Ern\H{o} Frenkel for bringing the
article by Roos \cite{Ro15} to our attention, to Gennadiy Averkov for
bringing the articles \cite{Ma48} and \cite{Ka88} to our attention,
and to Oliver Knill, Emanuel Lazar, and Steve Shnider for reading
several preliminary versions of the text.

\section{Funding}

Mikhail Katz is partially supported by the BSF grant 2020124 and the
ISF grant 743/22.



\begin{thebibliography}{A}


\bibitem{Ba98} Babenko, I.; Katz, M.\, Systolic freedom of orientable
  manifolds.  \emph{Annales scientifiques de l'Ecole normale
    sup\'erieure} \textbf{31} (1998), no.\;6, 787--809.



\bibitem{Ba93} Banaszczyk, W.\, New bounds in some transference
  theorems in the geometry of numbers.  \emph{Mathematische Annalen}
  \textbf{296} (1993), no.\;4, 625--635.


\bibitem{Ba95} Banaszczyk, W.\, Inequalities for convex bodies and
  polar reciprocal lattices in~$\R^n$.  \emph{Discrete \&
    Computational Geometry} \textbf{13} (1995), 217--231.



\bibitem{Ba09} Bangert, V; Katz, M.; Shnider, S.; Weinberger,
  S.\,~$E_7$, Wirtinger inequalities, Cayley 4-form, and homotopy.
  \emph{Duke Math. J}. \textbf{146} (2009), no.\,1, 35--70.



\bibitem{Ba86} Bavard, C.\, In\'egalit\'e isosystolique pour la
  bouteille de Klein.  \emph{Math. Ann}.  \textbf{274} (1986), no.\;3,
  439--441.


\bibitem{Dr08} Dranishnikov, A.; Katz, M.; Rudyak, Y.\, Small values
  of the Lusternik--Schnirelmann category for manifolds.
  \emph{Geom. Topol}.  \textbf{12} (2008), no.\;3, 1711--1727.


\bibitem{Dr11} Dranishnikov, A.; Katz, M.; Rudyak, Y.\, Cohomological
  dimension, self-linking, and systolic geometry.  \emph{Israel
    J. Math}.  \textbf{184} (2011), 437--453.




\bibitem{Fe65} Federer, H.\, Some theorems on integral currents.
  \emph{Trans. Amer. Math. Soc}.  \textbf{117} (1965), 43--67.


\bibitem{Fe1} Federer, H.\, \emph{Geometric Measure Theory}.
  Grundlehren der mathematischen Wissenschaften, \textbf{153}.
  Springer-Verlag, Berlin, 1969.


\bibitem{Fe2} Federer, H.\, Real flat chains, cochains, and
  variational problems.  \emph{Indiana Univ.\ Math.\ J.}  \textbf{24}
  (1974), 351--407.


\bibitem{Gr81} Gromov, M.\, \emph{Structures m\'etriques pour les
  vari\'et\'es riemanniennes}.  Edited by J. Lafontaine and P. Pansu.
  Textes Math\'ematiques \textbf{1} (1981).  CEDIC, Paris.


\bibitem{Gr2} Gromov, M.\, Systoles and intersystolic inequalities,
  Actes de la Table Ronde de G\'{e}om\'{e}trie Diff\'{e}rentielle
  (Luminy, 1992), 291--362, {\em S\'{e}min. Congr.}, \textbf{1},
  Soc. Math. France, Paris, 1996.
  \url{www.emis.de/journals/SC/1996/1/ps/smf_sem-cong_1_291-362.ps.gz}


\bibitem{Gr99} Gromov, M.\, \emph{Metric structures for Riemannian and
  non-Riemannian spaces}.  Based on the 1981 French original With
  appendices by M. Katz, P. Pansu and S. Semmes.  Translated from the
  French by Sean Michael Bates.  Progress in Mathematics,
  \textbf{152}.  Birkh\"auser Boston, Boston, MA, 1999.

\bibitem{Ja22} Jabbour, A.; Sabourau, S.\, Sharp upper bounds on the
  length of the shortest closed geodesic on complete punctured spheres
  of finite area.  \emph{Rev. Mat. Iberoam}.  \textbf{38} (2022),
  no.\;4, 1051--1065.


\bibitem{Jo48} John, F.\, Extremum problems with inequalities as
  subsidiary conditions. Studies and Essays Presented to R. Courant on
  his 60th Birthday, January 8, 1948, 187--204. Interscience
  Publishers, New York, 1948.

\bibitem{Ka88} Kannan, R.; Lov\'asz, L.\, Covering minima and
  lattice-point-free convex bodies.  \emph{Ann. of Math. (2)}
  \textbf{128} (1988), no.\;3, 577--602.


\bibitem{Ka08} Katz, M.\, Systolic inequalities and Massey products in
  simply-connected manifolds.  \emph{Israel J. Math}.  \textbf{164}
  (2008), 381--395.


\bibitem{Ka05} Katz, M.; Lescop, C.\, Filling area conjecture, optimal
  systolic inequalities, and the fiber class in abelian covers.
  Geometry, spectral theory, groups, and dynamics, 181--200,
  Contemp. Math., 387, Israel Math. Conf. Proc., Amer. Math. Soc.,
  Providence, RI, 2005.



\bibitem{Ka06} Katz, M.; Rudyak, Y.\, Lusternik-Schnirelmann category
  and systolic category of low-dimensional manifolds.
  \emph{Comm. Pure Appl. Math}.  \textbf{59} (2006), no.\,10,
  1433--1456.


\bibitem{Ka06b} Katz, M.; Sabourau, S.\, Hyperelliptic surfaces are
  Loewner.  \emph{Proc. Amer. Math. Soc}.  \textbf{134} (2006),
  no.\,4, 1189--1195.



\bibitem{Ka15} Katz, M.; Sabourau, S.\, Dyck's surfaces, systoles, and
  capacities.  \emph{Trans. Amer. Math. Soc}.  \textbf{367} (2015),
  no.\;6, 4483--4504.



\bibitem{Ka24a} Katz, M.; Sabourau, S.\, Logarithmic systolic growth
  for hyperbolic surfaces in every genus.  \emph{Proceedings of the
  American Mathematical Society} (2024).
  \url{https://arxiv.org/abs/2407.02041}

  \bibitem{Ka24b} Katz, M.; Sabourau, S.\, Nonpositively curved surfaces
  are Loewner (2024).  Journal of Geometric Analysis (2024).
  \url{https://doi.org/10.1007/s12220-024-01732-4},
  \url{https://arxiv.org/abs/2404.00757}



\bibitem{Ma48} Mahler, K.\, On lattice points in polar reciprocal
  convex domains.  \emph{Nederl. Akad. Wetensch., Proc}.  \textbf{51}
  (1948), 176--179 = Indagationes Math. 51, 482--485.


\bibitem{Mi86} Milman, V.; Schechtman, G.\, \emph{Asymptotic theory of
  finite-dimensional normed spaces. With an appendix by
  M. Gromov}. Lecture Notes in Mathematics, 1200.  Springer-Verlag,
  Berlin, 1986.



\bibitem{Pa99} Pansu, P.\, Profil isop\'erim\'etrique, m\'etriques
  p\'eriodiques et formes d'\'equilibre des cristaux.  \emph{ESAIM
    Control Optim. Calc. Var}.  \textbf{4} (1999), 631--665.


\bibitem{Pu52} Pu, P. M.\, Some inequalities in certain nonorientable
  Riemannian manifolds.  \emph{Pacific J. Math}.  \textbf{2} (1952),
  55--71.


\bibitem{Ro15} Roos, B.\, On Bobkov's approximate de Finetti
  representation via approximation of permanents of complex
  rectangular matrices.  \emph{Proc. Amer. Math. Soc}.  \textbf{143}
  (2015), no.\;4, 1785--1796.


\bibitem{Wh57} Whitney, H.\, \emph{Geometric integration theory}.
  Princeton University Press, Princeton, N. J., 1957.




\end{thebibliography}
\end{document}